\documentclass{amsart}
\usepackage{amssymb, epsfig}
\usepackage{pb-diagram}
\newtheorem{thm}{Theorem}[section]
\newtheorem{cor}[thm]{Corollary}
\newtheorem{lem}[thm]{Lemma}
\newtheorem{prop}[thm]{Proposition}

\theoremstyle{definition}
\newtheorem{defn}[thm]{Definition}

\theoremstyle{remark}

\numberwithin{equation}{section}

\renewcommand{\bibname}{\sc}
\renewcommand{\and}{\textnormal{and}\ }
\newenvironment{proof*}{\begin{proof}}{\end{proof}}

\newcommand{\cla}{\mathcal{A}}
\newcommand{\clg}{\mathcal{G}}
\newcommand{\clf}{\mathcal{F}}
\newcommand{\clc}{\mathcal{C}}

\newcommand{\clk}{\mathcal{K}}

\newcommand{\clr}{\mathcal{R}}

\newcommand{\bbz}{\mathbb{Z}}
\newcommand{\bbr}{\mathbb{R}}
\newcommand{\bbq}{\mathbb{Q}}

\newcommand{\bbn}{\mathbb{N}}
\newcommand{\bbk}{\mathbb{K}}
\newcommand{\bbs}{\mathbb{S}}
\newcommand{\bl}{B\ell}
\newcommand{\la}{\langle}
\newcommand{\ra}{\rangle}
\newcommand{\G}{\Gamma}
\newcommand{\gu}{\Gamma^U}
\newcommand{\ru}{\mathcal{R}^U}
\def\rank{\mbox{rank}}

\def\wti{\widetilde}

\DeclareMathOperator{\Hom}{Hom}

\DeclareMathOperator{\Rep}{Rep}

\title%[Splittings of von Neumann rho--invariants]
{Splittings of von Neumann rho--invariants of knots}
\author{Se-Goo Kim}
\address{Department of Mathematics and Research Institute for Basic
Sciences, Kyung Hee University, Seoul 130--701, Korea}
\email{sgkim@khu.ac.kr}
\urladdr{web.khu.ac.kr/\~{}sekim}

\author{Taehee Kim}
\address{Department of Mathematics, Konkuk
University, Seoul 143--701, Korea}
\email{tkim@konkuk.ac.kr} \urladdr{konkuk.ac.kr/\~{}tkim}
\def\subjclassname{\textup{2000} Mathematics Subject Classification}
\expandafter\let\csname subjclassname@1991\endcsname=\subjclassname
\expandafter\let\csname subjclassname@2000\endcsname=\subjclassname
\subjclass{Primary 57M25; Secondary 57N70}
\keywords{Knot, Concordance, Splitting, von Nuemann rho--invariant}
\date{\today}

\begin{document}

\begin{abstract}
We give a sufficient condition under which vanishing property of Cochran--Orr--Teichner knot concordance obstructions splits under connected sum. The condition is described in terms of self-annihilating submodules with respect to higher-order Blanchfield linking forms. This extends results of Levine and the authors on distinguishing knots with coprime Alexander polynomials up to concordance. As an application, we show that the knots constructed by Cochran, Orr and Teichner as the first examples of nonslice knots with vanishing Casson--Gordon invariants are not concordant to any knot of genus~1.
\end{abstract}

\maketitle

%%%%%%%%%%%%%%%%%% Section %%%%%%%%%%%%%%%%%%%%%

\section{Introduction}

Knots $K$ and $J$ in the 3--sphere $S^3$ are said to be \emph{concordant} if  there exists a (topologically) locally flat properly embedded annulus $S^1\times [0,1]$  in $S^3\times [0,1]$ whose boundary is the union of $K\times\{0\} \subset S^3\times\{0\}$ and $-J\times\{1\}\subset S^3\times\{1\}$. Here $-J$ denotes the mirror image of $J$ with reversed string orientation. A knot which is concordant to the unknot is said to be \emph{slice}. Equivalently, a knot is slice if and only if it bounds a locally flat 2--disk in the 4--ball. It is well-known that $K$ is concordant to $J$ if and only if the connected sum $K\#(-J)$ is slice. Concordance classes, which are the equivalence classes of knots modulo concordance, form an abelian group under connected sum, and this abelian group is called \emph{the knot concordance group $\clc$}. In the group $\clc$, the identity is the class of slice knots, and the inverse of the concordance class of $K$ is that of $-K$. 

In the 1960's, Levine  \cite{lev69b} defined \emph{the algebraic knot concordance group $\clg$} using Seifert forms of knots, and showed that there is a surjective homomorphism $\clc\to \clg$. A knot is \emph{algebraically slice} if it is mapped to the identity in $\clg$, or equivalently, if it has a metabolic Seifert form.
Using characters on the first homology of prime power order cyclic covers of the 3--sphere branched along a knot, Casson and Gordon \cite{cg86} exhibited the examples of nonslice knots which are algebraically slice.

It is known that in some cases we can distinguish knots up to concordance if the knots have coprime Alexander polynomials. Levine \cite{lev69a} proved that if the connected sum of two knots with coprime Alexander polynomials is algebraically slice, then so are both knots. In other words, if two knots are not algebraically slice and have coprime Alexander polynomials, then they are not algebraically concordant to each other, and hence they are not concordant. Regarding Casson--Gordon invariants, the first author showed a similar result in \cite{kim05}: if the connected sum of two knots with coprime Alexander polynomials has vanishing Casson--Gordon invariants, then so do both knots.

In \cite{cot03}, Cochran, Orr and Teichner revealed more structures of the group $\clc$ by constructing the filtration $\{\clf_n\}$ of $\clc$ which is indexed by nonnegative half-integers:
$$
    0\subset \cdots\subset \clf_{n.5}\subset \clf_n\subset\cdots\subset \clf_{1.5}\subset \clf_1 \subset \clf_{0.5}\subset \clf_0\subset \clc.
$$
For each $n\in \frac12\bbn_0$, where $\bbn_0=\bbn\cup\{0\}$, the group $\clf_n$  is the subgroup of \emph{$(n)$--solvable} knots. Roughly speaking, a knot $K$ is \emph{$(n)$--solvable} if the 0--surgery on $K$ in $S^3$, denoted $M_K$, bounds a certain $4$--manifold called an \emph{$(n)$--solution}, which can be considered as an approximation of a slice disk complement of order $n$ (see Definition~\ref{def:solvability}). The filtration $\{\clf_n\}$ reflects classical concordance invariants at low levels: a knot is algebraically slice if and only if it is $(0.5)$--solvable. If a knot is $(1.5)$--solvable, then all topological concordance invariants prior to their work in  \cite{cot03}, including Casson--Gordon invariants, vanish.

For each nonnegative integer $n$, Cochran, Orr and Teichner constructed obstructions for a knot being $(n.5)$--solvable using von Neumann $\rho$--invariants \cite[Theorem 4.2]{cot03} (also see Theorem~\ref{thm:rho invaiants}(1)). Note that slice knots are $(n)$--solvable for all $n\in \frac12\bbn_0$, and therefore these are also obstructions to a knot being slice. A basic idea of the obstructions in \cite[Theorem 4.2]{cot03} is as follows. For an $(n.5)$--solvable knot $K$, if a representation $\phi\colon \pi_1(M_K)\to \G$ for a group $\G$ satisfying certain conditions extends to an $(n.5)$--solution for $K$, then the von Neumann $\rho$--invariant $\rho(M_K,\phi)$ associated to $\phi$ vanishes. Therefore, it is essential to find representations on $\pi_1(M_K)$ that extend to $(n.5)$--solutions (or slice disk complements). For instance, to get the desired representations one can use certain quotient groups of the fundamental group of a putative $(n.5)$--solution for $K$. This method was introduced in \cite{ct07} and used by many authors.

On the other hand, the original idea in \cite{cot03} for getting representations that extend to an $(n.5)$--solution is to use \emph{rationally universal representations}. In \cite{cot03}, for each nonnegative integer $n$, they defined \emph{the rationally universal group $\G_n$}  and representations $\phi_n\colon \pi_1(M_K)\to \G_n$ via higher-order Blanchfield linking forms (see Section~\ref{subsec:universal} for more details). They also showed that if a knot $K$ is $(n.5)$--solvable, then some of these rationally universal representations $\phi_n$ extend to $(n.5)$--solutions for $K$ and the von Neumann $\rho$--invariants associated to such $\phi_n$ vanish \cite[Theorem 4.6]{cot03} (or see Definition~\ref{def:vanishing} and Theorem~\ref{thm:4.6}). In this paper, we say that a knot has \emph{vanishing $\rho$--invariants of order $n$} if it satisfies a criterion in \cite[Theorem 4.6]{cot03}. Therefore, Theorem~4.6 in \cite{cot03} can be rephrased that an $(n.5)$--solvable knot has vanishing $\rho$--invariants of order $n$. Although it is not easy to show that an arbitrarily given $(n.5)$--solvable knot $K$ has vanishing $\rho$--invariants of order $n$, these obstructions obtained using rationally universal representations have an advantage that they are described in terms of information only from the 3--manifold $M_K$, not using information from a putative $(n.5)$--solution for $K$, which is a 4--manifold bounded by $M_K$.
We remark that using these obstructions the first examples of nonslice knots with vanishing Casson-Gordon invariants were found \cite{cot03}. 

Regarding the approach for knot concordance using coprimeness of Alexander polynomials, the authors showed that if the connected sum of two knots with coprime Alexander polynomials has vanishing $\rho$--invariants of order 1, then so do both knots  \cite[Theorem 1.1]{kk08}. The aim of this paper is to generalize this splitting property of $\rho$--invariants of order 1 to  all orders $n\ge 1$, and obtain obstructions to knots being concordant. Our results are described in terms of higher-order Alexander modules and Blanchfield linking forms, which are explained below.

For a knot $K$, let  $\cla_0^{K}$ denote the (rational) Alexander module $H_1(M_K;\bbq[t^{\pm 1}])$ and $\bl_0^K$ the Blanchfield linking form defined on $\cla_0^K$. For knots $K$ and $J$, let $L=K\#  J$. It is well-known that $\cla_0^L$ splits as $\cla_0^L = \cla_0^{K} \oplus \cla_0^{J}$, and similarly we have $\bl_0^L= \bl_0^{K}\oplus \bl_0^{J}$. 
Generalizing $\cla_0^K$ and $\bl_0^K$ to higher-orders, for each positive integer $n$ and a representation $\phi_n\colon (M_K)\to \G_n$,  the \emph{$n$-th order Alexander module $\cla_n^K$ associated to $\phi_n$} and the \emph{$n$-th order Blanchfield linking form  $\bl_n^K$ associated to $\phi_n$} were defined in  \cite{cot03}. We note that the construction of $\phi_n$ depends on an inductive choice of elements $x_i\in \cla_i^K$ for $0\le i< n$, and hence so does that of $\cla_n^K$ and $\bl_n^K$.
In contrast to the case of $n=0$, for $n>0$ the module $\cla_n^L$ and the form $\bl_n^L$ do not split as a direct sum in general. Notheless, in Theorem~\ref{thm:splitting H_1} we give a criterion for $\phi_n\colon \pi_1(M_L) \to \G_n$ under which $\cla_n^L$ and $\bl_n^L$ associated to $\phi_n$ split as a direct sum. Loosely speaking, they split if $\phi_n$ is constructed by choosing $x_i$ `along $K$' for $0\le i < n$, which means that we have $x_i\in \cla_i^{K}\oplus 0\subset  \cla_i^{K} \oplus \cla_i^{J} =\cla_i^L$ where splitting of $\cla_i^L$ as a direct sum can be assumed for $i<n$  from the construction of $\phi_n$.  See Theorem~\ref{thm:splitting H_1} for more details.  Similarly, if we choose $x_i$ `along $J$', then $\cla_n^L$ and $\bl_n^L$ split as a direct sum. 

In \cite[Theorem 1.1]{kk08}, coprimeness of the Alexander polynomials of $K$ and $J$ was needed only to guarantee that every self-annihilating submodule $P_0^L$ of $\cla_0^L$ splits as $P_0^L=P_0^{K}\oplus P_0^{J}$ where $P_0^K$ and $P_0^J$ are self-annihilating submodules of $\cla_0^K$ and $\cla_0^J$, respectively. Recall that a $\bbq[t^{\pm 1}]$-submodule $P_0^K$ of $\cla_0^K$ is \emph{self-annihilating with respect to $\bl_0^K$} if $P_0^K = (P_0^K)^\perp$, where
\[
(P_0^K)^\perp:=\{x\in\cla_0^K\,|\, \bl_0^K(x,y) = 0 \mbox{ for all }y\in P_0^K\}.
\]
We generalize the condition of having coprime Alexander polynomials to higher-orders: for each nonnegative integer $n$ and  a knot $L$ where $L=K\# J$, we define the knot $L$ to have \emph{splitting self-annihilating submodules of order $n$ for $K$}, roughly speaking, if every self-annihilating submodule $P_n^L$ of $\cla_n^L$ splits as $P_n^L= P_n^{K}\oplus P_n^{J}$ (where $P_n^K$ and $P_n^J$ are self-annihilating submodules of $\cla_n^K$ and $\cla_n^J$, respectively) whenever $\phi_n\colon \pi_1(M_L)\to \G_n$ is constructed by choosing $x_i$ `along $K$' for $0\le i\le n-1$ in such a way that each $x_i$ is contained in a self-annihilating submodule of $\cla_i^K$. See Definition~\ref{def:splitting sas} for the precise definition. We remark that when $n>0$, this condition is not symmetric with respect to $K$ and $J$. 

We generalize the coprimeness condition to higher-orders as described above for technical reasons. Instead of the condition for higher-orders defined above, one might want to use the condition of having coprime \emph{higher-order Alexander polynomials}. However, compared to the case of $n=0$, we get more technical difficulties for using higher-order Alexander polynomials since higher-order Alexander polynomials are defined in \emph{noncommutative rings}. For example, for a noncommutative ring $R$ and $a\in R$, the element $a$ does not annihilates the right $R$-module $R/aR$ in general. Also note that when $n=0$, in the case of classical Alexander polynomials and Blanchfield linking forms, the condition of having splitting self-annihilating submodules is weaker than that of having coprime Alexander polynomials.

Now we are ready to state our main theorem.
\begin{thm}\label{main}
Let $n$ be a positive integer. Let $K$ and $J$ be knots such that $K\# J$ has splitting self-annihilating submodules of order $n-1$ for $K$. If $K\# J$ has vanishing $\rho$--invariants of order $n$, then so does $K$.
\end{thm}

\noindent Note that the case $n=1$ in Theorem~\ref{main} results in Theorem~1.1 in \cite{kk08}.  

Since an $(n.5)$--solvable knot has vanishing $\rho$--invariants of order $n$, Theorem~\ref{main} can be used as a concordance obstruction.
As an application of Theorem~\ref{main}, we show the following: in \cite{cot03, cot04}, Cochran, Orr and Teichner gave the first examples of knots which are $(2)$--solvable but not $(2.5)$--solvable. These knots are also the first examples of nonsilce knots with vanishing Casson-Gordon invariants. Let $K$ be one of these knots. Recall that \emph{the concordance genus of $K$} is defined to be the minimum of the genus of $K'$ among all knots $K'$ concordant to $K$. In Theorem~\ref{thm:example}, we show that $K$ is not concordant to any knot whose Alexander polynomial has  degree less than or equal to two, which implies that $K$ is not concordant to any knot of genus 1. Since $K$ has genus 2,  this shows that $K$ has concordance genus 2. Note that these are the first examples of knots that are $(2)$--solvable that are proved to be of concordance genus greater than 1.

We remark that in \cite{chl11} Cochran, Harvey and Leidy defined a refinement of the filtration $\{\clf_n\}$ via sequences of (Alexander) polynomials, and using it they distinguished, up to concordance, knots which have coprime Alexander polynomials or coprime higher-order Alexander polynomials. They also showed that the aforementioned knots in \cite{cot03,cot04} are not concordant to their knots in \cite{chl09}, which are not restricted to genus one knots. 

Our results are related to the following more general question which motivated our work: if two nonslice knots have coprime Alexander polynomials, are they mutually nonconcordant? One can see that the answer is no in smooth concordance due to the existence of knots with trivial Alexander polynomial that are not smoothly slice. In topological concordance, however, the answer is unknown to the authors' knowledge. Theorem~\ref{main} above and the work of Levine, Cochran--Harvey--Leidy, and the authors in \cite{lev69a, chl09, kim05, kk08} can be considered as evidence that the answer might be yes. In \cite{chl09} one can also find interesting results along this direction which involve the quotients of successive terms of the filtration $\{\clf_n\}$ and higher-order Alexander polynomials. 
Also regarding the question above, one may ask whether or not $K\# J$ has splitting self-annihilating submodules of order $n$ for $K$ (and for $J$) for all $n> 0$ if $K$ and $J$ have coprime Alexander polynomials, but the authors do not know the answer. 

This paper is organized as follows. In Section~\ref{sec:preliminaries}, we give background materials such as the notion of $(n)$--solvability, rationally universal groups, higher-order Alexander modules and Blanchfiled linking forms. Section~\ref{sec:vanishing} is devoted to the proof of Theorem~\ref{main}. We give examples in Section~\ref{sec:examples}. In this paper, unless mentioned otherwise all manifolds are assumed to be connected and oriented.

%%%%%%%%%% Section %%%%%%%%%%%%

\section{Preliminaries}
\label{sec:preliminaries}

In this section, we briefly review work of Cochran, Orr and Teichner on the filtration $\{\clf_n\}$ of $\clc$ in \cite{cot03} to have the paper self-contained. We claim no originality in this section except Subsection~\ref{subsec:vanishing}. 

\subsection{Filtration of the knot concordance group}\label{subsec:filtration}

Let $G$ be a group and $n$ a nonnegative integer. Then the \emph{$n$--th derived
group of $G$}, denoted $G^{(n)}$, is defined inductively as follows:
$G^{(0)}=G$ and $G^{(k)}=[G^{(k-1)},G^{(k-1)}]$ for $k\ge 1$.

Let $M$ be a closed 3--manifold and $W$ a spin 4--manifold such that $\partial W = M$. Let $\pi=\pi_1(W)$. For $n\ge 0$, we have the intersection form
$$
\lambda_n\colon H_2(W;\bbz[\pi/\pi^{(n)}])\times H_2(W;\bbz[\pi/\pi^{(n)}]) \to \bbz[\pi/\pi^{(n)}].
$$
Also there is the self-intersection invariant $\mu_n$. Refer to \cite[Section 7]{cot03} for the definition of $\mu_n$.

\begin{defn}\label{def:solvability}
Let $n$ be a nonnegative integer. A closed 3--manifold $M$ is \emph{$(n)$--solvable via $W$} if there exists a spin 4--manifold $W$ with $\partial W=M$ satisfying the following: 
    \begin{enumerate}
    \item $H_1(M;\bbz)\to H_1(W;\bbz)$ is an isomorphism,
    \item for $r=\frac12\rank_\bbz H_2(W;\bbz)$ and $1\le i,j\le r$, there exist elements $x_i$ and $y_j$  in $H_2(W;\bbz[\pi/\pi^{(n)}])$ such that $\lambda_n(x_i,x_j) = \mu_n(x_i) = 0$ and $\lambda_n(x_i,y_j) = \delta_{ij}$ for all $i,j$.
    \end{enumerate}
Then $W$ is said to be an \emph{$(n)$--solution} for $M$. The 3--manifold $M$ is \emph{$(n.5)$--solvable via $W$}, and $W$ is called an \emph{$(n.5)$--solution} for $M$, if $W$ satisfies the following additional condition:
    \begin{itemize}
    \item[(3)] for $1\le i \le r$, there exist elements $\wti{x_i}$  mapping to $x_i$ under the canonical homomorphism $H_2(W;\bbz[\pi/\pi^{(n+1)}])\to H_2(W;\bbz[\pi/\pi^{(n)}])$ such that $\lambda_{n+1}(\wti{x_i},\wti{x_j}) = \mu_{n+1}(\wti{x_i}) = 0$ for all $i,j$.
    \end{itemize}
For a knot $K$ and $m\in\frac12 \bbn_0$, we say that $K$ is \emph{$(m)$--solvable via $W$} if $M_K$, the 0--surgery  on $K$ in $S^3$, is $(m)$--solvable via $W$. In this case, $W$ is said to be an \emph{$(m)$--solution for $K$}. We denote by $\clf_m$ the set of all (concordance classes of) $(m)$--solvable knots.
\end{defn}

In \cite{cot03}, Cochran, Orr and Teichner  showed that $\clf_m\subset \clf_n$ if $m\ge n$,  and they also showed that slice knots are $(m)$--solvable for all $m\in \frac12 \bbn_0$. There are other variations of the notion of $(m)$--solvability such as rational $(m)$--solvability which uses rational coefficients instead of integer coefficients in its definition. It is easy to see that $(m)$--solvable knots are rationally $(m)$--solvable.

\subsection{Rationally universal groups and higher-order invariants}\label{subsec:universal}
For a nonnegative integer $n$, we define \emph{rationally universal groups} $\G_n$ as follows.

\begin{defn}\cite[Definition 3.1]{cot03}\label{def:ru groups}
We define \emph{the 0-th rationally universal group $\G_0$} to be  the infinite cyclic group $\bbz$ generated by $\mu$. For a nonnegative integer $n$, we let
\[
\clr_n=(\bbq\G_n)(\bbq[\G_n,\G_n]-\{0\})^{-1}
\]
and define \emph{the $(n+1)$-st rationally universal group $\G_{n+1}$} to be $\clk_n/\clr_n\rtimes \G_n$, where $\clk_n$ is the right ring of quotients of $\bbq\G_n$. That is, the ring $\clk_n$ is the (skew) quotient field of $\bbz\G_n$.
\end{defn}

In \cite{cot03}, these $\G_n$ and $\clr_n$ are denoted by $\gu_n$ and $\ru_n$, respectively. Note that $\clk_n$ are equipped with involutions induced from taking $g\mapsto g^{-1}$ for $g\in \G_n$. We denote the involution of $a\in \clk_n$ by $\bar{a}$.
One can easily see that $\G_n$ is $(n)$--solvable, that is, $(\G_n)^{(n+1)} = 0$, and that $\G_n$ embeds into $\clr_n$. Furthermore, the groups $\G_n$ are poly-torsion-free-abelian (abbreviated PTFA), that is, each $\G_n$ admits a normal series of finite length whose successive quotient groups are torsion-free abelian.

Note that $\clr_0=\bbq\G_0=\bbq[\mu^{\pm 1}]$ and $\clk_0=\bbq(\mu)$, where $\bbq(\mu)$ is the quotient field of $\bbq[\mu^{\pm 1}]$. For higher order cases, we have the following proposition.

\begin{prop}\cite[Corollary 3.3]{cot03}\label{prop:ru}
For $n\ge 0$, the rings $\clr_n$ are left and right principal ideal domains and isomorphic to (skew) Laurent polynomial rings $\bbk_n[\mu^{\pm 1}]$, where $\bbk_n$ are the (skew) quotient fields of $\bbz[\G_n,\G_n]$. In particular, the (skew) fields $\clk_n$ are isomorphic to $\bbk_n(\mu)$.
\end{prop}

We explain how to obtain representations $\phi_n\colon \pi_1(M)\to \G_n$ for a closed 3--manifold $M$ with first Betti number 1. The example for $M$ we should keep in mind is $M_K$, the 0--surgery on a knot $K$ in $S^3$. To this end, we define higher-order Alexander modules and Blanchfield linking forms.

We recall the classical invariants. Let $\phi_0\colon \pi_1(M)\to \G_0 = \la \mu \ra$ be the homomorphism induced by the abelianization. For example, if $M$ is the 0--surgery $M_K$, then we can choose $\phi_0$ to be the abelianization sending a meridian of a knot $K$ to $\mu$.
Then we obtain a coefficient system $\phi_0\colon \pi_1(M)\to \G_0\hookrightarrow\clr_0$, and from this we obtain
$H_1(M;\clr_0)$ as a (right) $\clr_0$--module, which is \emph{the 0-th order Alexander module $\cla_0^M$} of $K$. This Alexander module is the same as the classical rational Alexander module. It is well-known that we have
the Blanchfield form $\bl_0:\cla_0^M\times \cla_0^M\to \clk_0/\clr_0$, which is
nonsignular in the sense that it provides an isomorphism $\cla_0^M\cong
\Hom_{\clr_0}(\cla_0^M,\clk_0/\clr_0)$. Now for a nonnegative integer $n$, suppose that we have a representation $\phi\colon \pi_1(M)\to \G_n$. Then \emph{the $n$-th order Alexander module of $M$ associated to $\phi$}, denoted $\cla_n^{M,\phi}$, is defined to be the first homology $H_1(M;\clr_n)$ with the twisted coefficient system given by $\phi\colon \pi_1(M)\to \G_n \hookrightarrow \clr_n$. We write $\cla_n^{K,\phi}$ for $\cla_n^{M,\phi}$ when $M=M_K$. We suppress superscripts when they are understood from the context.

\begin{thm}\cite[Proposition 2.11, Theorem 2.13]{cot03}\label{thm:Blanchfield form} Let $M$ be a closed 3--manifold with first Betti number 1. If $\phi\colon \pi_1(M)\to \G_n$ is a nontrivial representation, then $\cla_n$ is a (right) $\clr_n$-torsion module, and there is a nonsingular symmetric linking form $\bl_n^{M,\phi}\colon \cla_n\times\cla_n \to \clk_n/\clr_n$.

\end{thm}

\noindent The linking form $\bl_n^{M,\phi}$ is \emph{the $n$-th order Blanchfield form on $\cla_n$}. When $n=0$, this is the same as the classical rational Blanchfield form. We write $\bl_n^{K,\phi}$ for $\bl_n^{M,\phi}$ when $M=M_K$, and sometimes suppress superscripts for simplicity. We describe how $\bl_n$ is defined: for $[x],[y]\in \cla_n$, there exists $r\in \clr_n$ such that $xr = \partial c$ for some 2-chain $c\in C_2(M;\clr_n)$. Then $\bl_n([x],[y]) = \frac{1}{r}(c\cdot y)$ where $c\cdot y$ is the equivariant intersection over $\clr_n$.

Now we explain how to obtain a representation $\phi_n\colon \pi_1(M)\to \G_n$ assuming we are given a representation $\phi_{n-1}\colon \pi_1(M)\to \G_{n-1}$. Recall that $\G_n=\clk_{n-1}/\clr_{n-1}\rtimes \G_{n-1}$. We define $\Rep_{\G_{n-1}}^\ast(\pi_1(M),\G_n)$ to be the set of representations $\phi_n\colon \pi_1(M)\to \G_n$ modulo $\clk_{n-1}/\clr_{n-1}$-conjugations such that $p\circ \phi_n = \phi_{n-1}$ where $p\colon \G_n \to \G_{n-1}$ is the canonical projection. In the theorem below, we show how to obtain representations $\phi_n\colon \pi_1(M)\to \G_n$ by inductively choosing elements from Alexander modules.

\begin{thm}\cite[Theorem 3.5]{cot03}\label{thm:extension} Let $M$ be a closed 3--manifold with first Betti number 1. Let $\phi_0\colon \pi_1(M)\to \G_0$ be the abelianization. Suppose that $n\ge 1$ and we have constructed a representation $\phi_{n-1}:\pi_1(M)\to\G_{n-1}$, the $(n-1)$-st order Alexander module $\cla_{n-1}$ associated to $\phi_{n-1}$, and the $(n-1)$-st order Blanchfield form $\bl_{n-1}\colon \cla_{n-1}\times \cla_{n-1}\to \clk_{n-1}/\clr_{n-1}$.  Then the following hold.
\begin{enumerate}
\item There is a bijection $f\colon H^1(M;\clk_{n-1}/\clr_{n-1})\longleftrightarrow \Rep_{\G_{n-1}}^\ast(\pi_1(M),\G_n)$ which is natural with respect to continuous maps.
\item There is an isomorphism $\cla_{n-1} \cong H^1(M;\clk_{n-1}/\clr_{n-1})$ with $f$ giving a natural bijection $\tilde{f}\colon \cla_{n-1} \longleftrightarrow \Rep_{\G_{n-1}}^\ast(\pi_1(M),\G_n)$. In particular, each choice of an element $x\in \cla_{n-1}$ induces a representation $\phi_n\colon \pi_1(M)\to \G_n$ such that $p\circ \phi_n = \phi_{n-1}$.  This $\phi_n$ will also be denoted by $\phi_{n-1,x}$ to indicate its dependency of $x$. 
\item If $x\in \cla_{n-1}$, then the image of $x$ under the map
$$
\cla_{n-1}\xrightarrow{\cong} H^1(M;\clk_{n-1}/\clr_{n-1})\to \Hom_{\clr_{n-1}}(\cla_{n-1},\clk_{n-1}/\clr_{n-1})
$$
is the map given by $y\mapsto \bl_{n-1}(x,y)$.
In particular, for $\phi_n=\phi_{n-1,x}$ and $y\in \pi_1(M)$ with $\phi_{n-1}(y) = 0$ (hence $y$ can be considered as an element in $\cla_{n-1}$), if $\bl_{n-1}(x, y) = 0$, then $\phi_n(y) = 0.$
\end{enumerate}
\end{thm}
For the representation $\phi_n=\phi_{n-1,x}$ constructed as in Theorem~\ref{thm:extension}(2), we say that $\phi_n$ is \emph{associated to $\phi_{n-1}$ and $x$}. Theorem~\ref{thm:extension}(2) implies that for $0\le i\le n-1$, by inductively choosing $x_i\in \cla_i=\cla_i(x_0,x_1,\ldots,x_{i-1})$, we can obtain a representation $\phi_n\colon \pi_1(M)\to \G_n$, which is denoted by $\phi_{0,x_0,x_1,\ldots, x_{n-1}}$, or more simply, $\phi_{x_0,x_1,\ldots, x_{n-1}}$. In this case, we also say that $\phi_n$, $\cla_n^{\phi_n}$, and $\bl_n^{\phi_n}$ are \emph{associated to $x_0,x_1,\ldots, x_{n-1}$}.

For a later use, we consider the case that $\phi_n$ is obtained by taking $\phi_n = i\circ \phi_{n-1}$ where $i\colon \G_{n-1}\hookrightarrow \G_n=\clk_{n-1}/\clr_{n-1}\rtimes \G_{n-1}$ is the canonical injection to the right summand, that is, $\phi_n$ is associated to $\phi_{n-1}$ and $0\in \cla_{n-1}$. In this case, the coefficient system $\phi_n\colon \pi_1(M)\to \clr_n$ factors through $\phi_{n-1}\colon \pi_1(M)\to \clr_{n-1}$ and we have the following theorem.

\begin{thm}\cite[Theorem 5.16]{cha06}\label{thm:factoring through}
Let $n$ be a positive integer. If $\phi_n\colon \pi_1(M)\to \G_n$ factors through $\phi_{n-1}\colon \pi_1(M)\to \G_{n-1}$ as above, then the following hold.
    \begin{enumerate}
    \item $\cla_n = \cla_{n-1}\otimes_{\clr_{n-1}} \clr_n$ as (right) $\clr_n$-modules.
    \item For the canonical homomorphism $i\colon \clk_{n-1}/\clr_{n-1}\to \clk_n/\clr_n$ and $x,y\in \cla_{n-1}$ and $a,b\in \clr_n$,
    $$
    \bl_n(x\otimes a, y\otimes b) = a\cdot i(\bl_{n-1}(x,y))\cdot \bar{b}.
    $$
    \end{enumerate}
\end{thm}

Now if $\phi_n\colon \pi_1(M)\to \G_n$ factors through $\phi_0\colon \pi_1(M)\to \G_0$, then by applying Theorem~\ref{thm:factoring through} repeatedly we obtain that $\cla_n=\cla_0\otimes_{\clr_0} \clr_n$ and $\bl_n(x\otimes a, y\otimes b) = a\cdot i(\bl_0(x,y))\cdot \bar{b}$ where $i\colon \clk_0/\clr_0\to \clk_n/\clr_n$ is the canonical homomorphism.
If $\phi_n\colon \pi_1(M)\to \G_n$ factors through $\phi_0\colon \pi_1(M)\to \G_0$, then we write $\bl_n^{M, \phi_0}$ for $\bl_n^{M,\phi_n}$.

\subsection{Obstructions to a knot being $(n.5)$--solvable}\label{subsec:obstructions}
Let $M$ be a closed $3$--manifold, and let $\phi\colon \pi_1(M)\to \G$ be a representation where $\G$ is a countable discrete group. Then one can define the von Neumann $\rho$--invariant $\rho(M,\phi)\in \bbr$ of $M$ associated to $\phi$ (see \cite{chg85}). If $\phi$ extends to $\bar{\phi}\colon \pi_1(W)\to \Gamma$ for a compact oriented $4$--manifold $W$ with $\partial W=M$, then the invariant $\rho(M,\phi)$ can be computed as $\rho(M,\phi) = \sigma^{(2)}(W, \bar{\phi}) - \sigma_0(W)$ where $\sigma^{(2)}(W,\bar{\phi})$ denotes the $L^{(2)}$--signature of the intersection form on $H_2(W;\bbz\Gamma)$ and $\sigma_0(W)$ the ordinary signature of $W$. (The existence of such $\bar{\phi}$ and $W$ can be assumed for the computation of $\rho(M,\phi)$ by Theorem~\ref{thm:rho invaiants}(2) below and Weinberger's work.) For a knot $K$ and a representation $\phi\colon \pi_1(M_K) \to \G$, we write $\rho(K,\phi)$ for $\rho(M_K,\phi)$. Recall that $(n)$--solvable 3--manifolds (or knots) are rationally $(n)$--solvable.

\begin{thm}\cite[Theorem 4.2]{cot03}\cite[Section 2]{cot04}\label{thm:rho invaiants} Let $M$ be a closed 3--manifold and $\G$ a group.
    \begin{enumerate}
    \item Let $\G$ be an $(n)$--solvable PTFA group. If $M$ is (rationally) $(n.5)$--solvable via $W$ to which the coefficient system $\phi\colon \pi_1(M) \to \G$ extends, then $\rho(M,\phi) = 0$.
    \item (Subgroup property) If $\phi\colon \pi_1(M)\to \G$ factors through $\phi'\colon \pi_1(M)\to \G'\subset \G$, then $\rho(M,\phi) = \rho(M,\phi')$.
    \item For a knot $K$, if $\phi\colon \pi_1(M_K)\to \bbz$ is a nontrivial homomorphism, then $\rho(K,\phi)$ equals the integral of the Levine-Tristram signature function $\sigma_\omega(K)$ of the knot $K$, integrated over the circle normalized to length one. In addition, if $K$ is algebraically slice, then $\rho(K,\phi)=0$.
    \item If $\phi\colon \pi_1(M)\to \G$ is a trivial homomorphism, then $\rho(M,\phi)=0$.
    \end{enumerate}
\end{thm}

Cochran, Orr and Teichner investigated in which cases representations to $\G_n$ extend to $(n.5)$--solutions, and obtained their obstructions to a knot being $(n.5)$--solvable, which are given below. Recall that an $\clr_n$--submodule $P$ of $\cla_n$ is said to be \emph{self-annihilating with respect to $\bl_n$} if $P=P^\perp$ where \[
P^\perp :=\{x\in \cla_n\mid \bl_n(x,y)=0 \mbox{ for all } y\in \cla_n\}.
\]
Definition~\ref{def:vanishing} and Theorem~\ref{thm:4.6} below can be found in \cite[Theorem 4.6]{cot03}.

\begin{defn}\label{def:vanishing}
Let $n$ be a positive integer. A knot $K$ has \emph{vanishing $\rho$--invariants of order $n$} if it satisfies the following conditions $(1)$--$(n)$.

\begin{itemize}
\item[($1$)] There exists a self-annihilating submodule $P_0$ of $\cla_0$ such that for each  $x_0\in P_0$ and the corresponding representation $\phi_1=\phi_{0,x_0}\colon \pi_1(M_K)\to \G_1$, the invariant $\rho(K,\phi_1)=0$.

\item[($2$)] For each $x_0\in P_0$, there exists a self-annihilating submodule $P_1=P_1(x_0)$ of $\cla_1 = \cla_1(x_0)$ such that for each $x_1\in P_1$ and the corresponding representation $\phi_2=\phi_{1,x_1}\colon \pi_1(M_K)\to \G_2$, the invariant $\rho(K,\phi_2)=0$.

\item[($3$)] For each $x_1\in P_1=P_1(x_0)$ with $x_0\in P_0$, there exists a self-annihilating submodule $P_2=P_2(x_0,x_1)$ of $\cla_2 = \cla_2(x_0,x_1)$ such that for each $x_2\in P_2$ and the corresponding representation $\phi_3=\phi_{2,x_2}\colon \pi_1(M_K)\to \G_3$, the invariant $\rho(K,\phi_3)=0$.

\centerline{\vdots}

\item[($n$)] For each $x_{n-2}\in P_{n-2}=P_{n-2}(x_0,x_1,\ldots, x_{n-3})$ with $x_i\in P_i$, $0\le i\le n-3$, there exists a self-annihilating submodule $P_{n-1} = P_{n-1}(x_0,x_1,\ldots, x_{n-2})$ of $\cla_{n-1} = \cla_{n-1}(x_0,x_1,\ldots, x_{n-2})$ such that for each $x_{n-1}\in P_{n-1}$ and the corresponding representation $\phi_n=\phi_{n-1,x_{n-1}}\colon \pi_1(M_K)\to \G_n$, the invariant $\rho(K,\phi_n)=0$.
\end{itemize}

\end{defn}

Below we state the obstruction to a knot being $(n.5)$--solvable.
\begin{thm}\cite[Theorem 4.6]{cot03}\label{thm:4.6}
Let $n$ be a positive integer. If a knot $K$ is (rationally) $(n.5)$--solvable, then $K$ has vanishing $\rho$--invariants of order $n$.
\end{thm}

\subsection{Vanishing system of submodules}
\label{subsec:vanishing}

We give another description of the condition of having vanishing $\rho$--invariants of order $n$.  For a knot $K$ and a submodule $P$ of $\cla_i^{K,\phi_i}$, we say that $P$ is \emph{associated to $\phi_i$}.

\begin{defn}\label{def:vanishing 2}
Let $n$ be a positive integer. For a knot $K$, let $\bbs_n$ be a nonempty set of $n$-tuples $(P_0, P_1, \ldots, P_{n-1})$ where each $P_i$ is a self-annihilating submodule of $\cla_i^{K,\phi_i}$ associated to a representation $\phi_i\colon \pi_1(M_K)\to \G_i$ for $0\le i\le n-1$. The set $\bbs_n$ is a \emph{vanishing system of submodules of order $n$ for $K$} if for each $(P_0, P_1, \ldots, P_{n-1})\in \bbs_n$ the following hold.
\begin{enumerate}
	\item For each $i$ such that $1\le i\le n-1$, the representation $\phi_i$ is associated to $\phi_{i-1}$ and some $x_{i-1}\in P_{i-1}$.
	\item If $x\in P_i$ for some $i\le n-2$, then there exists $(P_0',P_1',\ldots, P_{n-1}')\in \bbs_n$ with $P_j'$ associated to $\phi_j'\colon \pi_1(M_K)\to \G_j$ for $0\le j\le n-1$ such that the following hold:
	\begin{enumerate}	
		\item	for $0\le j\le i$, we have $\phi_j'=\phi_j$ and $P_j'=P_j$, 
		\item the representation $\phi_{i+1}'$ is associated to $\phi_i$ and $x$.
	\end{enumerate}
	\item For $0\le i\le n-1$, the invariant $\rho(K,\phi_i) = 0$. Furthermore, for each $x_{n-1}\in P_{n-1}$ and $\phi_n=\phi_{n-1,x_{n-1}}\colon \pi_1(M_K)\to \G_n$, the invariant $\rho(K,\phi_n)=0$.
\end{enumerate}
\end{defn}

\noindent It is easy to see the following proposition, of which proof is omitted.
\begin{prop}\label{prop:vanishing}
Let $n$ be a positive integer. A knot $K$ has vanishing $\rho$--invariants of order $n$ if and only if it has a vanishing system of submodules of order $n$.
\end{prop}

%%%%%%%%%% Section %%%%%%%%%%%%

\section{Vanishing of $\rho$--invariants}\label{sec:vanishing}
In this section, we investigate $\rho$--invariants for a knot which is given as the connected sum of a pair of knots, and prove Theorem~\ref{main}. Let $K$ and $J$ be knots in $S^3$, and let $L=K\# J$. We denote by $M_K$ the 0--surgery on $K$ in $S^3$, and  denote by $M_J$ and
$M_L$ the 0--surgery on $J$ in $S^3$ and the 0--surgery on $L$ in $S^3$, respectively. Throughout this section, unless mentioned otherwise, the symbol $M$ denotes $M_L$. We write $E_K$ for $S^3 - N(K)$, the exterior of $K$, and write $E_J$ and $E_L$ for those of $J$ and $L$, respectively. Let us denote meridians of $K$, $J$, and $L$ by $m_K$, $m_J$, and $m_L$, respectively. We also write $\ell_K$, $\ell_J$, and $\ell_L$ for longitudes of $K$, $J$, and $L$, respectively. We abuse notations so that they also represent their homology classes and homotopy classes as well.
 For a group $G$ and elements $g_1, g_2, \ldots, g_m\in G$, we denote the normal subgroup of $G$ generated by $g_i$, $1\le i\le m$, by $\la g_1, g_2,\ldots, g_m\ra$.

\subsection{Splitting of higher-order Alexander modules and Blanchfield forms}\label{subsec:splitting H_1}
We aim to find certain conditions under which the vanishing property of $\rho$--invariants splits via connected sum operation. In doing so, we need to verify that certain representations $\pi_1(M)\to \G_i$, $0\le i\le n$, induce representations $\pi_1(M_K)\to \G_i$ so that we can define higher-order Alexander modules and $\rho$--invariants of $K$. The following lemma is well-known. For instance, see~\cite{kk08}.

\begin{lem}\label{lem:pi_1 of M}
$\pi_1(M)/\la\ell_K\ra = \pi_1(M)/\la\ell_J\ra = \pi_1(M_K)\ast\pi_1(M_J)/\la m_K  m_J^{-1}\ra$.
\end{lem}
\begin{proof} 
Note the following equalities:
    \begin{align*}
    \pi_1(M)/\la \ell_K\ra & = \pi_1(E_K)*\pi_1(E_J)/\la  m_K m_J^{-1},\,\ell_K\ell_J,\, \ell_K\ra \\
    & = \pi_1(E_K)*\pi_1(E_J)/\la  m_K m_J^{-1},\,\ell_K,\, \ell_J\ra \\
    & = \pi_1(M_K)\ast\pi_1(M_J)/\la m_K  m_J^{-1}\ra.
    \end{align*}
One can prove that $\pi_1(M)/\la\ell_J\ra = \pi_1(M_K)\ast\pi_1(M_J)/\la m_K  m_J^{-1}\ra $ similarly.
\end{proof}

Therefore, if we have a representation $\phi_i\colon \pi_1(M)\to \G_i$ for some $i$ with $\phi_i(\ell_J) = 0$, then one can obtain a representation $\phi_i^K\colon\pi_1(M_K)\to \G_i$ by taking compositions $\pi_1(M_K)\hookrightarrow \pi_1(M_K)\ast\pi_1(M_J)/\la m_K  m_J^{-1}\ra \xrightarrow{\cong} \pi_1(M)/\la\ell_J\ra \xrightarrow{\phi_i}\G_i$.  
Similarly, for a given representation $\phi_i\colon \pi_1(M)\to \G_i$ for some $i$ with $\phi_i(\ell_J) = 0$, the representation $\phi_i^J\colon \pi_1(M_J)\to \G_i$ is induced from $\phi_i$.   We note the restrictions of $\phi_i^K$ to $E_K$ is the same as that of $\phi_i$ to $E_K$, that is, $\phi_i ^K|_{\pi_1(E_K)}=\phi_i |_{\pi_1(E_K)}$. Since $\phi_i(\ell_J) = \phi_i(\ell_K) = 0$, the homomorphism $\phi_i |_{\pi_1(E_K)}$ extends to $\pi_1(M_K)$, and this extension of $\phi_i |_{\pi_1(E_K)}$ to $\pi_1(M_K)$ is the same as $\phi_i^K$, and similarly for $M_J$ and $\phi_i^J$. Henceforth we drop the superscripts from $\phi_i^K$ and $\phi_i^J$ if they are clearly understood in the context.

For $\phi_0\colon \pi_1(M)\to \G_0$, since $\ell_J\in\pi_1(M)^{(2)}$ and $\G_0$ is $(0)$--solvable, we have $\phi_0(\ell_J)=0$. Therefore, the representations $\phi_0$ are induced for $\pi_1(M_K)$ and $\pi_1(M_J)$, and it is well-known that $\cla_0^L = \cla_0^{K} \oplus \cla_0^{J}$ and $\bl_0^L = \bl_0^{K}\oplus \bl_0^{J}$. But this direct sum decomposition does not hold in general for higher-order cases. In fact, $\dim_{\bbk_n} \cla_n^L\ge \dim_{\bbk_n} \cla_n^{K} + \dim_{\bbk_n} \cla_n^{J}$ and sometimes the inequality becomes strict. Nonetheless, we can obtain the direct sum decompositions $\cla_n^L=\cla_n^{K} \oplus \cla_n^{J}$ and $\bl_n^L = \bl_n^{K}\oplus \bl_n^{J}$ if the elements $x_i$ are inductively chosen from the summands corresponding to $K$ while constructing representations $\pi_1(M)\to \G_n$. This procedure is stated in the following theorem.

\begin{thm}\label{thm:splitting H_1}
Under the notations introduced above, the following hold.
\begin{itemize}
    \item[(0)] Let $\phi_0\colon \pi_1(M)\to \G_0$ be the abelianization. Then the following hold.
        \begin{itemize}
        \item[--] $\phi_0(\ell_J)=0 $, and we have induced representations $\phi_0$ on $\pi_1(M_K)$ and $\pi_1(M_J)$.
        \item[--] $\cla_0^L = \cla_0^{K}\oplus \cla_0^{J}$ and $\bl_0^L = \bl_0^{K}\oplus \bl_0^{J}$.
        \end{itemize}
    \item[(1)] Let $x_0\in \cla_0^{K,\phi_0}\oplus 0\subset \cla_0^{L,\phi_0}$ and $\phi_1\colon \pi_1(M_L)\to \G_1$ be the representation associated to $x_0$.  Then the following hold:
        \begin{itemize}
        \item[--] $\phi_1(\ell_J)=0$, and we have induced representations $\phi_1$ on $\pi_1(M_K)$ and $\pi_1(M_J)$.
        \item[--] $\phi_1\colon \pi_1(M_J)\to \G_1$ factors through $\phi_0$, and hence $\cla_1^{J}=\cla_0^{J}\otimes_{\clr_0} \clr_1$ and $\bl_1^{J,\phi_1} = \bl_1^{J,\phi_0}$.
        \item[--] $\cla_1^L = \cla_1^{K}\oplus \cla_1^{J}$ and $\bl_1^L = \bl_1^{K}\oplus \bl_1^{J}$.
         \end{itemize}
    \item[$(2)$] Let $x_1\in \cla_1^{K,\phi_1}\oplus 0\subset \cla_1^{L,\phi_1}$ and $\phi_2\colon \pi_1(M_L)\to \G_2$ be the representation associated to $x_0, x_1$.  Then the following hold:
        \begin{itemize}
        \item[--] $\phi_2(\ell_J)=0$, and we have induced representations $\phi_2$ on $\pi_1(M_K)$ and $\pi_1(M_J)$.
        \item[--] $\phi_2\colon \pi_1(M_J)\to \G_2$ factors through $\phi_0$, and hence $\cla_2^{J}=\cla_0^{J}\otimes_{\clr_0} \clr_2$ and $\bl_2^{J,\phi_2} = \bl_2^{J,\phi_0}$.
        \item[--] $\cla_2^L = \cla_2^{K}\oplus \cla_2^{J}$ and $\bl_2^L = \bl_2^{K}\oplus \bl_2^{J}$.
        \end{itemize}
\centerline{\vdots}
    \item[$(n)$] Let $x_{n-1}\in \cla_{n-1}^{K,\phi_{n-1}}\oplus 0\subset \cla_{n-1}^{L,\phi_{n-1}}$ and $\phi_n\colon \pi_1(M_L)\to \G_n$ be the representation associated to $x_0, x_1,\ldots,x_{n-1}$.  Then the following hold:
        \begin{itemize}
        \item[--] $\phi_n(\ell_J)=0$, and we have induced representations $\phi_n$ on $\pi_1(M_K)$ and $\pi_1(M_J)$.
         \item[--] $\phi_n\colon \pi_1(M_J)\to \G_n$ factors through $\phi_0$, and hence $\cla_n^{J}=\cla_0^{J}\otimes_{\clr_0} \clr_n$ and $\bl_n^{J,\phi_n} = \bl_n^{J,\phi_0}$.
        \item[--] $\cla_n^L = \cla_n^{K}\oplus \cla_n^{J}$ and $\bl_n^L = \bl_n^{K}\oplus \bl_n^{J}$.
        \end{itemize}
\centerline{\vdots}

\end{itemize}
\end{thm}

\begin{proof}
We use an induction argument. The statement $(0)$ is a well-known classical result. For $n>0$, assuming the statements $(1)$ through $(n-1)$ hold, we prove that the statement $(n)$ holds.

For $x_{n-1}=(x_{n-1}^K,0)\in \cla_{n-1}^{K,\phi_{n-1}}\oplus 0\subset \cla_{n-1}^{L,\phi_{n-1}}$, the representation $\phi_n = \phi_{n-1,x_{n-1}}$ is defined by Theorem~\ref{thm:extension}. Since $\phi_{n-1}(\ell_J)=0$ by the assumption, we can consider $\ell_J$ as an element in $\cla_{n-1}^L$. By the assumption that $\cla_{n-1}^L = \cla_{n-1}^{K}\oplus \cla_{n-1}^{J}$, we can express $\ell_J$ as $(0,\ell_J)\in \cla_{n-1}^{K}\oplus \cla_{n-1}^{J}$. Then $\bl_{n-1}^L(x_{n-1},\ell_J) = \bl_{n-1}^{K}(x_{n-1}^K,0) + \bl_{n-1}^{J}(0,\ell_J) = 0$, and by Theorem~\ref{thm:extension}(3) we have $\phi_n(\ell_J)=0$.

We show that $\phi_n(y) = 0$ for all $y\in \pi_1(M_J)^{(1)}$, which will imply that $\phi_n\colon \pi_1(M_J)\to \G_n$ factors through $\phi_0$. By Theorem~\ref{thm:factoring through}, this will also show that $\cla_n^{J}=\cla_0^{J}\otimes_{\clr_0} \clr_n$ and $\bl_n^{J,\phi_n} = \bl_n^{J,\phi_0}$.  Let $y\in \pi_1(M_J)^{(1)}$. Then we can regard $y$ as an element in $\cla_0^{J}$. Since $\cla_0^L = \cla_0^{K}\oplus \cla_0^{J}$ and $\bl_0^L = \bl_0^{K}\oplus \bl_0^{J}$ and $x_0\in \cla_0^{K}$, it follows that $\bl_0^L(x_0,y) = 0$. Then by Theorem~\ref{thm:extension}(3), $\phi_1(y)=0$, and $y$ can be considered as an element in $\cla_1^{J}$. Using this argument repeatedly and by the assumption that the statements $(1)$--$(n-1)$ hold, we obtain that $\phi_n(y) =0$.

To prove that $\cla_n^L = \cla_n^{K}\oplus \cla_n^{J}$, we first show that $H_1(E_L;\clr_n)=H_1(E_K;\clr_n)\oplus H_1(E_J;\clr_n)$. Note that $E_L=E_K\cup E_J$ where the meridians  $ m_K$ and $ m_J$ are identified as the meridian $ m_L$. Then with the coefficient systems induced from $\phi_n$, we obtain the following exact sequence using a Mayer-Vietoris argument.
\begin{align*}
    \to H_1( m_L;\clr_n)& \to H_1(E_K;\clr_n)\oplus H_1(E_J;\clr_n)\to
    H_1(E_L;\clr_n) \\
    \to H_0( m_L;\clr_n)& \xrightarrow{i} H_0(E_K;\clr_n)\oplus H_0(E_J;\clr_n)\to
    H_0(E_L;\clr_n).
\end{align*}
Note that since $ m_L$ and $E_J$ are subspaces of $M_J$, the maps $\phi_n$ on $\pi_1( m_L)$ and $\pi_1(E_J)$ factor through $\phi_0$. Since Theorem~\ref{thm:factoring through}(1) works for any spaces as well, we have $H_\ast( m_L;\clr_n)=H_\ast( m_L;\clr_0)\otimes_{\clr_0} \clr_n$ and $H_\ast(E_J;\clr_n)=H_\ast(E_J;\clr_0)\otimes_{\clr_0} \clr_n$. In particular, $H_0( m_L;\clr_n)=H_0(E_J;\clr_n) = \clr_n/\{\phi_n(\alpha)x-x\mid \alpha\in \pi_1(m_L), x\in\clr_n\} = \clr_n/(\mu-1)\clr_n$ as right $\clr_n$-modules, and the map $i$ in the above exact sequence is injective. (In fact, if $n\ge 1$ and $\phi_{n}$ does not factor through $\phi_0$,  then $H_0(E_K;\clr_n)=0$  and the map $i$ is an isomorphism, but we do not need this fact in this paper.)
Since $ m_L\cong S^1$ and $\phi_0( m_L)\ne 0$,  we have $H_1( m_L;\clr_n)=0$. Thus $H_1(E_L;\clr_n)= H_1(E_K;\clr_n)\oplus H_1(E_J;\clr_n)$.

Using the above direct sum decomposition of $H_1(E_L;\clr_n)$ and the fact that $\ell_L=\ell_K\ell_J$, we can deduce that
\begin{align*}
    \cla_n^L & =H_1(M;\clr_n)\\
    & = H_1(E_L;\clr_n)/\la \ell_L\ra\\
    & = H_1(E_K;\clr_n)\oplus H_1(E_J;\clr_n)/\la (\ell_K,\ell_J)\ra
\end{align*}
Here the second equality follows from a Mayer-Vietoris Sequence. Moreover, since $H_1(E_J;\clr_n)=H_1(E_J;\clr_0)\otimes_{\clr_0} \clr_n$, we have $\ell_J=0$ in $H_1(E_J;\clr_n)$ and $H_1(E_J;\clr_n) = H_1(M_J;\clr_n)$. Therefore,
\begin{align*}
    \cla_n^L & = H_1(E_K;\clr_n)\oplus H_1(E_J;\clr_n)/\la (\ell_K,0)\ra\\
     &= H_1(E_K;\clr_n)/\la \ell_K\ra\oplus H_1(M_J;\clr_n)\\
     & = \cla_n^{K}\oplus \cla_n^{J}.
\end{align*}

Finally, we show that $\bl_n^L = \bl_n^{K}\oplus \bl_n^{J}$. Let $[x]$ and $[y]$ be classes in $H_1(M_L;\clr_n)$. Then $\bl_n^L([x],[y])=\frac{1}{r}(c\cdot y)$, where $c$ is a chain in
$C_2(M_L;\clr_n)$ such that $\partial c=xr$ for some $r\in \clr_n$ and
$c\cdot y$ is the equivariant intersection over $\clr_n$.
Using that
  $H_1(M;\clr_n)= H_1(M_K;\clr_n)\oplus H_1(M_J;\clr_n)$,
  $H_1(M;\clr_n)= H_1(E_L;\clr_n)/\langle \ell_L\rangle$,
  $H_1(M_K;\clr_n)= H_1(E_K;\clr_n)/\langle \ell_K\rangle$,
  and
  $H_1(M_J;\clr_n)\cong H_1(E_J;\clr_n)$, one can see that a class in $H_1(M;\clr_n)$ can be represented by a sum of chains supported by $E_K$ or $E_J$. Therefore we can obtain the desired splitting of $\bl_n^L$ as in the case of classical Blanchfield form.
\end{proof}

\subsection{Splitting of von Neumann $\rho$--invariants}\label{subsec:splitting rho}
In showing that a knot $K$ has vanishing $\rho$--invariants of higher-orders, we need existence of certain self-annihilating submodules of higher-order Alexander modules of $K$. Indeed, in \cite{kk08} the authors showed that $K$ has vanishing $\rho$--invariants of order 1 under the condition that $K$ and $J$ have coprime Alexander polynomials. This coprimeness condition was necessary only to guarantee that every self-annihilating submodule of $\cla_0^L$ can be decomposed as a direct sum of self-annihilating submodules of $\cla_0^{K}$ and $\cla_0^{J}$. Though we can define higher-order Alexander polynomials up to certain indeterminacy, over a noncommutative ring such as $\clr_n$ for $n>0$ it is not easy to use them. Moreover, eventually  we need splitting of self-annihilating submodules of $\cla_n^L$. Therefore, we choose to generalize the condition of coprimeness of classical Alexander polynomials to higher-order cases in terms of splitting of self-annihilating submodules. Since the construction of $\cla_n^L$ depends on the choice of $x_i\in \cla_i^L$, $0\le i\le n-1$, we need an assumption that for \emph{every} $\cla_n^L$ associated to certain $x_i$ for $0\le i\le n-1$,  \emph{every} self-annihilating submodule of $\cla_n^L$ is decomposed as a direct sum. We state this condition carefully below. We recall that $M=M_L$ where $L=K\# J$.

\begin{defn}\label{def:splitting sas}
Under the notations introduced above:
\begin{itemize}
    \item[(0)] Let $\phi_0\colon \pi_1(M)\to \G_0$ be the abelianization. Recall that by Theorem~\ref{thm:splitting H_1} we have $\cla_0^L = \cla_0^{K}\oplus \cla_0^{J}$ and $\bl_0^L = \bl_0^{K}\oplus \bl_0^{J}$.
        \begin{itemize}
        \item[--] A connected sum $L$ has \emph{splitting self-annihilating submodules of order 0 for $K$} if every self-annihilating submodule $P_0^L$ of $\cla_0^L$ is decomposed as $P_0^L=P_0^K\oplus P_0^J$ where $P_0^K$ and $P_0^J$ are self-annihilating submodules of $\cla_0^{K}$ and $\cla_0^{J}$.
        \end{itemize}

    \item[(1)] Let $L$ be a connected sum that has splitting self-annihilating submodules of order 0 for $K$.
    	\begin{itemize}
        \item[--] An element \emph{$x_0\in \cla_0^L$ belongs to $K$} if there exists a self-annihilating submodule $P_0^L$ of $\cla_0^L$ such that $x_0\in P_0^K\oplus 0\subset P_0^L$.
        \item[--] Recall that if $\phi_1\colon \pi_1(M)\to \G_1$ is a representation associated to $x_0$ which belongs to $K$, then by Theorem~\ref{thm:splitting H_1} we have $\cla_1^L = \cla_1^{K}\oplus \cla_1^{J}$ and $\bl_1^L = \bl_1^{K}\oplus \bl_1^{J}$. A connected sum $L$ has \emph{splitting self-annihilating submodules of order 1 for $K$} if for every $\phi_1$ associated to some $x_0$ which belongs to $K$, every self-annihilating submodule $P_1^L$ of $\cla_1^L$ is decomposed as $P_1^L=P_1^K\oplus P_1^J$ where $P_1^K$ and $P_1^J$ are self-annihilating submodules of $\cla_1^{K}$ and $\cla_1^{J}$.
        \end{itemize}

  \centerline{\vdots}

    \item[$(n)$] Let $L$ be a connected sum that has splitting self-annihilating submodules of order $n-1$ for $K$. 
        \begin{itemize}
        \item[--] For $0\le i\le n-2$, let $x_i\in \cla_i^L$ be elements which belong to $K$. For $x_{n-1}\in \cla_{n-1}^L$ where $\cla_{n-1}^L$ is associated to $x_0,\ldots, x_{n-2}$, the elements \emph{$x_0,\ldots,x_{n-1}$ belong to $K$} if there exists a self-annihilating submodule $P_{n-1}^L$ of $\cla_{n-1}^L$ such that $x_{n-1}\in P_{n-1}^K\oplus 0\subset P_{n-1}^L$.
        \item[--] Recall that if $\phi_n\colon \pi_1(M)\to \G_n$ is a representation associated to $x_0,\ldots, x_{n-1}$ which belong to $K$, then by Theorem~\ref{thm:splitting H_1} we have $\cla_n^L = \cla_n^{K}\oplus \cla_n^{J}$ and $\bl_n^L = \bl_n^{K}\oplus \bl_n^{J}$. A connected sum $L$ has \emph{splitting self-annihilating submodules of order $n$ for $K$} if for every $\phi_n$ associated to some $x_0,\ldots,x_{n-1}$ which belong to $K$, every self-annihilating submodule $P_n^L$ of $\cla_n^L$ is decomposed as $P_n^L=P_n^K\oplus P_n^J$ where $P_n^K$ and $P_n^J$ are self-annihilating submodules of $\cla_n^{K}$ and $\cla_n^{J}$.
        
        \end{itemize}
  \centerline{\vdots}

\end{itemize}

\end{defn}

Note that $L$ has splitting of self-annihilating submodules of order $0$ for $K$ and $J$ if $K$ and $J$ have coprime Alexander polynomials (for example, see \cite{kk08}). Now we are ready to give a proof of Theorem~\ref{main}. 

\begin{proof}[Proof of Theorem~\ref{main}]
Recall that we let $L=K\# J$ and $M=M_L$. By Proposition~\ref{prop:vanishing}, the knot $L$ has a vanishing system of submodules of order $n$, denoted $\bbs_n$. Using $\bbs_n$, we will construct a vanishing system of submodules of order $n$ for $K$, which will complete the proof.

Let $P=(P_0,P_1,\ldots, P_{n-1})\in \bbs_n$, where for $1\le i\le n-1$ each $P_i$ is a self-annihilating submodule of $\cla_i^{M,\phi_i}$ such that $\phi_i$ is associated to $\phi_{i-1}$ and $x_{i-1}\in P_{i-1}$. If the elements $x_0, x_1, \ldots, x_{n-1}$ belong to $K$, we say that \emph{$P$ belongs to $K$}. Now if $P$ belongs to $K$, then by Theorem~\ref{thm:splitting H_1}, for $0\le i\le n-1$ we obtain the representations $\phi_i^K\colon \pi_1(M_K)\to \G_i$ and $\phi_i^J\colon \pi_1(M_J)\to \G_i$ which are induced from $\phi_i$, and we have $\cla_i^L = \cla_i^{K}\oplus \cla_i^{J}$ and $\bl_i^L = \bl_i^{K}\oplus \bl_i^{J}$. Furthermore, since $L$ has splitting self-annihilating submodules of order $n-1$ for $K$, for $0\le i\le n-1$ we have $P_i = P_i^K\oplus P_i^J$ where $P_i^K$ and $P_i^J$ are self-annihilating submodules of $\cla_i^{K}$ and $\cla_i^{J}$, respectively. For each $P\in \bbs_n$ which belongs to $K$, let $P^K = (P_0^K,P_1^K,\ldots, P_{n-1}^K)$. Now we let $\bbs_n^K=\{P^K\,\mid\, P\in \bbs_n \mbox{ and } P \mbox{ belongs to } K\}$.

We show that $\bbs_n^K$ is nonempty. Pick an arbitrary element $P=(P_0,P_1,\ldots, P_{n-1})$ in $\bbs_n$ where each $\phi_i$ is associated to $\phi_{i-1}$ and $y_{i-1}\in P_{i-1}$ for $1\le i\le n-1$. Since $L$ has splitting self-annihilating submodules of order $n-1$, we have $P_0=P_0^K\oplus P_0^J$. Pick $x_0\in P_0^K\oplus 0\subset P_0$. Since Property~(2) in Definition~\ref{def:vanishing 2} holds for $P$, after re-labeling if necessary, we may assume that $\phi_1$ is associated to $\phi_0$ and $x_0$. Since $x_0$ belongs to $K$, and since $L$ has splitting of self-annihilating submodules of order $n-1$, we have $P_1=P_1^K\oplus P_1^J$. Now pick $x_1\in P_1^K\oplus 0\subset P_1$. By using the argument above again, we may assume that $\phi_2$ is associated to $\phi_1$ and $x_1$,  and that $P_2=P_2^K\oplus P_2^J$. By repeating this argument as necessary, we can find $P\in \bbs_n$ such that each $\phi_i$ is associated to $\phi_{i-1}$ and $x_{i-1}\in P_{i-1}^K\oplus 0\subset P_{i-1}$ for $1\le i\le n-1$, which implies that $P$ belongs to $K$. 

We show that $\bbs_n^K$ is a vanishing system of submodules of order $n$ for $K$. We need to show that each $P^K$ in $\bbs_n^K$ satisfies Properties (1)--(3) in Definition~\ref{def:vanishing 2}. Let $P^K$ be an element in $\bbs_n^K$ as defined above. Since $x_0,\ldots, x_{n-1}$ belong to $K$, there exist $x_i^K\in P_i^K$ for $0\le i\le n-1$ such that $x_i=(x_i^K,0)\in P_i^K\oplus 0\subset \cla_i^K\oplus \cla_i^J$. Note that each $\cla_i^K$ is associated to $\phi_i^K$. Therefore, to show that $P^K$ satisfies Property~(1), it suffices  to show that each $\phi_i^K$ is associated to $\phi_{i-1}^K$ and $x_{i-1}^K$ for $1\le i\le n-1$. For every $x\in \cla_{i-1}^K$, since $\bl_{i-1}^L=\bl_{i-1}^K\oplus \bl_{i-1}^J$, we have $\bl_{i-1}^K(x_{i-1}^K,x)=\bl_{i-1}^L(x_{i-1},(x,0))$ where $(x,0)\in \cla_{i-1}^K\oplus \cla_{i-1}^J=\cla_{i-1}^L$. Also note that $\phi_i^K$ is induced from $\phi_i$ which is associated to $\phi_{i-1}$ and $x_{i-1}$ for $1\le i\le n-1$. Combining this observation with Theorem~\ref{thm:extension}, one can deduce that the representation $\pi_1(M_K)\to \G_i$ associated to $\phi_{i-1}^K$ and $x_{i-1}^K$ is the same as $\phi_i^K$.

We show that $P^K$ satisfies Property~$(2)$. Fix an integer $i$ such that $0\le i\le n-2$. Let $y_i^K\in P_i^K$, and let $y_i=(y_i^K,0)\in P_i^K\oplus P_i^J = P_i$. Since $P$ satisfies Property~(2), there exists $P' =(P_0',P_1',\ldots, P_{n-1}')\in \bbs_n$ with each $P_j'$ associated to some $\phi_j'\colon \pi_1(M)\to \G_j$ such that the following hold: 
\begin{itemize}
\item	[(a)] for $0\le j\le i$, we have $\phi_j'=\phi_j$ and $P_j'=P_j$, 
\item[(b)] the representation $\phi_{i+1}'$ is associated to $\phi_i$ and $y_i$.
\end{itemize}
In particular, we have $P_j'=(P_j')^K\oplus (P_j')^J$ where $(P_j')^K=P_j^K $ and $(P_j')^J=P_j^J$ for $j\le i$. Then since $L$ has splitting self-annihilating submodules of order $n-1$, and since $y_i$ belongs to $K$ by its definition, we have $P_{i+1}'=(P_{i+1}^K)'\oplus (P_{i+1}^J)'$ for some self-annihilating submodules $(P_{i+1}^K)'$ and $(P_{i+1}^J)'$ of $\cla_{i+1}^K$ and $\cla_{i+1}^J$. Now, choose an arbitrary element $y_{i+1}^K\in(P_{i+1}^K)'$, and we let $y_{i+1}=(y_{i+1}^K,0)\in (P_{i+1}^K)'\oplus 0\subset P_{i+1}'$. Then $y_{i+1}$ belongs to $K$. Using the argument above again with $y_{i+1}$, after re-labeling if necessary, we may assume that $P_{i+2}'$ splits as $P_{i+2}'=(P_{i+2}^K)'\oplus (P_{i+2}^J)'$. Now by repeating this argument as necessary, we may assume that for $i+1\le j\le n-1$ we also have $P_j' = (P_j')^K\oplus (P_j')^J$ and the representation $\phi_j'$ is associated to $\phi_{j-1}'$ and some $x_{j-1}' = ((x_{j-1}')^K,0)\in( P_{j-1}')^K\oplus (P_{j-1}')^J$ such that $x_i'=y_i$. (Recall that $P_j'$ are equal to $P_j$ for $0\le j\le i$.) Therefore, $P'$ belongs to $K$, and we obtain $(P')^K\in \bbs_n^K$ where $(P')^K = ((P_0')^K,(P_1')^K,\ldots, (P_{n-1}')^K)$.  For $(P')^K$, we have $(P_j')^K = P_j^K$ for $j\le i$ and $(\phi_{i+1}')^K$ is associated to $\phi_i^K$ and $y_i^K$, and hence Property (2) holds for $P^K$.

For Property (3), we construct a cobordism $C$ between the disjoint union
$M_K\cup M_J$ and $M$. This construction is well-known, and it is briefly described below. (For
details, refer to \cite[Section 4]{cot04}.)
\begin{figure}
  \setlength{\unitlength}{0.6pt}
  \begin{picture}(291,87)
    \put(0,0){\includegraphics[scale=0.6]{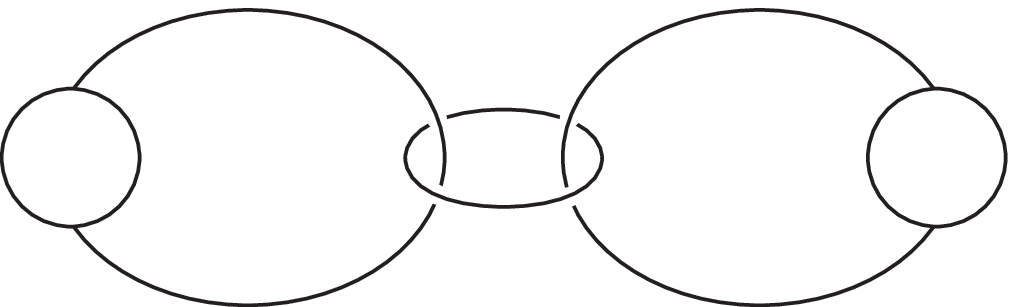}}
    \put(22,43){\makebox(0,0){$K$}}
    \put(273,43){\makebox(0,0){$J$}}
    \put(145,19){\makebox(0,0){$\eta$}}
    \put(145,62){\makebox(0,0)[b]{\small$0$}}
    \put(73,11){\makebox(0,0){\small$0$}}
    \put(219,11){\makebox(0,0){\small$0$}}
  \end{picture}
  \caption{}
  \label{fig:1}
\end{figure}
Attach a 1--handle at the top level
between $M_K\times [0,1]$ and $M_J\times [0,1]$ so that the upper boundary
is the 0--surgery on the split link $K \cup J$. Then attach a
2--handle along $\eta$ with zero framing as indicated in
Figure~\ref{fig:1}.  The resulting 4--manifold is the desired $C$. Its
boundary at the bottom, $\partial_{-}C$, is the disjoint union $-(M_K\cup
M_J)$ and the boundary at the top, $\partial_{+} C$, is $M$.  Also one can easily see that $\pi_1(C)= \pi_1(M_K)\ast \pi_1(M_J)/\langle  m_K m_J^{-1}\rangle$, and by Lemma~\ref{lem:pi_1 of M} we obtain that $\pi_1(C) = \pi_1(M)/\langle \ell_J\rangle$.

For $P\in \bbs_n$ which belongs to $K$ and $P^K\in \bbs_n^K$ as defined above, we need to show that $\rho(K,\phi_i^K)=0$ for $0\le i\le n-1$ and $\rho(K,\phi_n^K)=0$ for every representation $\phi_n^K\colon \pi_1(M_K)\to \G_n$ which is associated to $\phi_{n-1}^K$ and some $x_{n-1}^K\in P_{n-1}^K$.  Let  $i$ be an integer such that $0\le i\le n-1$. Since $\phi_i(\ell_J)=0$, the representation $\phi_i$ extends to $C$. Denote the extension of $\phi_i$ to $C$ by $\phi_i^C$.  Moreover, on $M_K\cup M_J$ which is a subspace of $C$, this extension $\phi_i^C$ agrees with the representations $\phi_i^K$ on $\pi_1(M_K)$ and $\phi_i^J$ on $\pi_1(M_J)$ induced from $\phi_i\colon \pi_1(M)\to \G_i$. Then we have
\[
\sigma^{(2)}(C, \phi_i^C) - \sigma_0(C) = \rho(L,\phi_i) - (\rho(K,\phi_i^K) + \rho(J,\phi_i^J)),
\]
where $\sigma^{(2)}(C,\phi_i^C)$ is the $L^{(2)}$-signature on $H_2(C;\bbz\G_i)$ and $\sigma_0(C)$ is the ordinary signature of $C$. By \cite[Lemma 4.2]{cot04}, we get $H_2(C;\clk_i)=H_2(C;\bbq)=0$, and therefore $\sigma^{(2)}(C,\phi_i^C)=\sigma_0(C)=0$. Since Property~(3) holds for $P$, we have $\rho(L,\phi_i)=0$. By the construction of $P$, the representation $\phi_i^J$ factors through $\phi_0$. By the assumption, the knot $J$ has a self-annihilating submodule of $\cla_0^{J}$, and therefore $J$ is algebraically slice. By Theorem~\ref{thm:rho invaiants}, it follows that $\rho(J,\phi_i^J)=0$. Therefore $\rho(K,\phi_i^K)=0$.

Let $\phi_n^K\colon \pi_1(M_K)\to \G_n$ be a representation which is associated to $\phi_{n-1}^K$ and some $x_{n-1}^K\in P_{n-1}^K$. Then $\phi_n^K$ is the same as the representation induced from $\phi_n\colon \pi_1(M)\to \G_n$ which is associated to $\phi_{n-1}$ and $x_{n-1}$ where $x_{n-1}=(x_{n-1}^K,0)\in P_{n-1}^K\oplus P_{n-1}^J = P_{n-1}$. Note that $\rho(L, \phi_n) = 0$ since Property~(3) holds for $P$. Then we can show that $\rho(K,\phi_n^K)=0$ similarly as above, and this shows that $P^K$ satisfies Property~(3).
\end{proof}

%%%%%%%% Section %%%%%%%%%%%%%%

\section{Examples}\label{sec:examples}
In this section, using Theorem~\ref{main} we show that the knots constructed by Cochran, Orr and Teichner in \cite{cot03, cot04}, which are $(2)$--solvable but not $(2.5)$--solvable, are not concordant to any knot whose Alexander polynomial has degree less than or equal to two. Since the knots of Cochran--Orr--Teichner have genus 2, this will imply that they are not concordant to genus 1 knots, and hence they have concordance genus 2.

We refer the reader to \cite{cot03, cot04} for the detailed construction of the knots of Cochran, Orr and Teichner. Instead, in the next lemma we list some properties of the knots which we will need later. Before listing the properties, we discuss some technical issues. Let $K$ be one of the knots of Cochran--Orr--Teichner. For each $x\in P_0^K\subset  \cla_0^{K}$ where $P_0^K$ is a self-annihilating submodule, there is a representation $\phi_1\colon \pi_1(M_K)\to \G_1=\clk_0/\clr_0\rtimes \G_0$ associated to $\phi_0$ and $x$. Using the arguments in \cite[p.~483]{cot03}, we can replace $\G_1$ by a subgroup $\G_1'$ of $\G_1$ such that $\G_1'$ contains the image of $\phi_1$ and $\G_1'=S\rtimes \G_0$ where $S\cong \bbq[t^{\pm 1}]/(t^2-3t+1)^m$ for some positive integer $m$. (This $\G_1'$ is denoted by $\G_1$ in \cite{cot03}. Also see \cite[p.~115]{cot04}.) That is, re-labeling $\G_1'$ by $\G_1$, Theorem~\ref{main} and related theorems in this paper are all available. Henceforth, in this section we use this re-labeled $\G_1$. We list properties of $K$ in Lemma~\ref{lem:cot knot property} below, which can be found in \cite{cot03, cot04}.  Recall from Proposition~\ref{prop:ru} that $\bbk_1$ is the quotient field of $\bbz[\G_1,\G_1]=\bbz[\bbq[t^{\pm 1}]/(t^2-3t+1)^m]$. 

\begin{lem}\cite{cot03,cot04}\label{lem:cot knot property} Let $K$ be one of the knots in \cite{cot03, cot04} which are $(2)$--solvable but not $(2.5)$--solvable.
\begin{enumerate}
    \item The Alexander polynomial $\Delta_K(t)$ is $(t^2-3t+1)^2$.
    \item As $\clr_0$-modules, we have $\cla_0^{K} = \bbq[t^{\pm 1}]/(t^2-3t+1)^2$, and $\cla_0^{K}$ has a unique nontrivial proper $\clr_0$-submodule, denoted $P_0^K$. The module $P_0^K$ is also a unique self-annihilating submodule of $\cla_0^{K}$.
    \item Let $\cla_1^{K,\phi_1}$ be associated to $x\in P_0^K$. Then as right $\clr_1$-modules, we have $\cla_1^{K,\phi_1}=\clr_1/(\mu-1)(\mu-k)\clr_1$ for some $k\in \bbk_1$, and  $\cla_1^{K,\phi_1}$ has a unique nontrivial proper $\clr_1$-submodule. This unique submodule is generated by $\mu-1$, and it is a unique self-annihilating submodule of $\cla_1^{K, \phi_1}$.
    \item The knot $K$ does not have vanishing $\rho$--invariants of order $2$.
\end{enumerate}
\end{lem}

We will also need the following proposition and corollary later. Recall that a ring is \emph{simple} if it has no nontrivial proper (two-sided) ideals. Also recall that by Proposition~\ref{prop:ru} we have $\clr_1=\bbk_1[\mu^{\pm 1}]$.

\begin{prop}
  \label{prop:simple}
  The ring $\clr_1$ is simple.
\end{prop}
\begin{proof}
Note that $\bbk_1$ is commutative. However, $\clr_1$ is noncommutative: note that $S= \bbq[t^{\pm 1}]/(t^2-3t+1)^m$ is a right $\bbz\Gamma_0$--module, and that the semi-direct product $\Gamma_1=\bbq[t^{\pm 1}]/(t^2-3t+1)^m\rtimes\Gamma_0$, which can be embedded into $\bbq(t)/\bbq[t^{\pm 1}]\rtimes \G_0$ as a submodule, is defined by choosing the right multiplication (see \cite[p.453]{cot03}). Therefore, for $a\in \bbk_1$ we have $a\mu = \mu \alpha(a)$ for the automorphism $\alpha:\bbk_1\to\bbk_1$ which is multiplication by $t$ on the right. Also observe that $\alpha^n$ is not the identity map for any $n>0$.

To show that $\clr_1$ is simple, we should show that $\clr_1$ has no nontrivial proper (two-sided) ideal. Suppose to the contrary: let $I$ be a nontrivial proper ideal of $\clr_1$. Since $\alpha$ is an automorphism, the ring $\clr_1$ is a PID. Since $\clr_1$ is a PID, we have $I=\langle f(\mu)\rangle$ for some nonzero $f(\mu)\in\clr_1$. In fact, the element $f(\mu)$ has  least degree among the elements in $I$ and we have $\langle f(\mu)\rangle = f(\mu)\clr_1 = \clr_1f(\mu)$. Since $I\ne \clr_1$, one can see that $f(\mu)$ is not a unit in $\clr_1$. Therefore the degree of $f(\mu)$ is greater than 0, and hence we can write $f(\mu)=a_n\mu^n+\cdots +a_m\mu^m$ where $n<m$, $a_n\ne 0$, and $a_m\ne 0$. Let $b$ be a nonzero element in $\bbk_1$. Since $f(\mu)\clr_1=\clr_1 f(\mu)$, there is $c\in\bbk_1$ such that $b f(\mu)=f(\mu) c$.
Therefore, we have $\alpha^n(a_n)\alpha^n(b)=\alpha^n(a_n)c$ and $\alpha^m(a_m)\alpha^m(b)=\alpha^m(a_m)c$. Since neither $a_n$ nor $a_m$ is 0, we obtain $\alpha^n(b) = \alpha^m(b)$. Since $b$ is arbitrary, it follows that $\alpha^n = \alpha^m$ and $\alpha^{m-n}$ is the identity map, which is a contradiction.
\end{proof}

The corollary below follows from Proposition~\ref{prop:simple} and \cite[Prop 2.8 in p.496]{coh85}.
\begin{cor}\label{cor:cyclic}
  Every finitely generated right torsion $\clr_1$--module is cyclic.
\end{cor}

The following theorem gives us the desired examples.
\begin{thm}\label{thm:example}
If $K$ is one of the knots in \cite{cot03, cot04} which are $(2)$--solvable but not $(2.5)$--solvable, then $K$ is not concordant to any knot whose Alexander polynomial has degree less than or equal to two. In particular, the knot $K$ has concordance genus 2.
\end{thm}
\begin{proof}
Let $J$ be a knot whose Alexander polynomial has degree less than or equal to two. If $J$ has trivial Alexander polynomial, then by Freedman's work $J$ is slice. Since $K$ is not $(2.5)$--solvable, it implies that $K$ is not concordant to $J$. From now on we assume that $J$ has degree 2 Alexander polynomial.

Suppose that $J$ is concordant to $K$. Since $K$ is $(2)$--solvable, so is $J$, and this implies that $J$ is algebraically slice (see \cite[Remark 1.3.2]{cot03}). Therefore, we may assume that the Alexander polynomial of $J$ is $\Delta_J(t) = (nt-(n-1))((n-1)t-n)$ for some integer $n\ge 2$. Since $K$ is concordant to $J$, the connected sum $K\# (-J)$ is slice. For convenience, we re-label the knot $-J$ by $J$. Note that this re-labeling does not change $\Delta_J(t)$.

Let $L=K\# J$. Now we have that $L$ is slice, and in particular it is $(2.5)$--solvable. By Theorem~\ref{thm:4.6}, the knot $L$ has vanishing $\rho$--invariants of order 2. We will show that $L$ has splitting self-annihilating submodules of order $1$ for $K$. By Theorem~\ref{main}, this will imply that $K$ has vanishing $\rho$--invariants of order $2$, which is a contradiction by Lemma~\ref{lem:cot knot property}(4), and it will complete the proof.

Henceforth, we use the notations given in the first paragraph of Section~\ref{sec:vanishing}.  Since $\Delta_K(t) = (t^2-3t+1)^2$ and $\Delta_J(t) = (nt-(n-1))((n-1)t-n)$ are coprime, by \cite[Theorem 3.1]{kk08} the knot $L$ has vanishing self-annihilating submodules of order 0 for $K$. That is, if $P_0$ is a self-annihilating submodule of $\cla_0^L$, then $P_0 = P_0^K\oplus P_0^J$ where $P_0^K$ and $P_0^J$ are self-annihilating submodules of $\cla_0^K$ and $\cla_0^J$, respectively. Note that $P_0^K$ is the unique nontrivial proper submodule of $\cla_0^K$ by Lemma~\ref{lem:cot knot property}(2).

Let $P_0$ be a self-annihilating submodule of $\cla_0^L$ and let $x_0\in P_0^K\oplus 0\subset P_0$. Let $\phi_1\colon \pi_1(M)\to \G_1 =\bbq[t^{\pm 1}]/(t^2-3t+1)^m\rtimes\Gamma_0$ be a representations associated to $\phi_0$ and $x_0$. Then $\cla_1^L = \cla_1^K\oplus \cla_1^J$ by Theorem~\ref{thm:splitting H_1}.

Let $P_1$ be a self-annihilating submodule of $\cla_1^L$ which is associated to $\phi_1$. Let $P_1^K=P_1\cap (\cla_1^K\oplus 0)$ and $P_1^J=P_1\cap(0\oplus \cla_1^J)$. We will show that $P_1=P_1^K\oplus P_1^J$, and that $P_1^K$ and $P_1^J$ are self-annihilating submodules of $\cla_1^K$ and $\cla_1^J$, respectively. This will imply that $L$ has vanishing self-annihilating submodules of order 1 for $K$.

It is obvious that $P_1^K\oplus P_1^J\subset P_1$. To show $P_1\subset P_1^K\oplus P_1^J$, we assume Lemma~\ref{lem:splitting P_1} below which asserts that $P_1^K$ is the unique nontrivial proper submodule of $\cla_1^K$ (see Lemma~\ref{lem:cot knot property}~(3)), which is generated by $\mu-1$. Note that Lemma~\ref{lem:splitting P_1}, combined with Lemma~\ref{lem:cot knot property}~(3), also implies that $P_1^K=(P_1^K)^\perp$. 

Let $(x,y)\in P_1\subset \cla_1^K\oplus \cla_1^J$. Then for every $x'\in P_1^K$, we have $\bl_1^L( (x,y),(x',0))=0$ since $(x',0)\in P_1$ and $P_1=P_1^\perp$. Since $\bl_1^L( (x,y),(x',0)) = \bl_1^K(x,x') + \bl_1^J(y,0) = \bl_1^K(x,x')$, we have $\bl_1^K(x,x')=0$ for all $x'\in P_1^K$, and hence $x\in (P_1^K)^\perp$. Since $P_1^K=(P_1^K)^\perp$, we have $x\in P_1^K$. By the definition of $P_1^K$, it follows that $(x,0)\in P_1$. Since $(0,y)=(x,y)-(x,0)$, we also have $(0,y)\in P_1$. This shows that if $(x,y)\in P_1$, then $x\in P_1^K$ and $y\in P_1^J$.
Therefore $P_1\subset P_1^K\oplus P_1^J$, and it follows that $P_1= P_1^K\oplus P_1^J$.

Now we show that $P_1^J=(P_1^J)^\perp$. It is clear that $P_1^J\subset (P_1^J)^\perp$ since $P_1^J\subset P_1$ and $P_1 = P_1^\perp$. Let $y\in (P_1^J)^\perp$. Then,
for every $(x',y')\in P_1^K\oplus P_1^J=P_1$, we have $\bl_1^L( (0,y),(x',y'))=\bl_1^K(0,x') + \bl_1^{J}(y,y')=\bl_1^{J}(y,y')$, and $\bl_1^{J}(y,y')=0$ since $y'\in P_1^J$. Therefore $(0,y)\in P_1^\perp=P_1$, and hence $y\in P_1^J$. This shows $(P_1^J)^\perp=P_1^J$.
\end{proof}

\begin{lem}\label{lem:splitting P_1}
The module $P_1^K$ defined in Theorem~\ref{thm:example} is the unique nontrivial proper submodule of $\cla_1^K$.
\end{lem}
\begin{proof}
Recall that $P_1$ is a self-annihilating submodule of $\cla_1^L$ and $\cla_1^L=\cla_1^K\oplus \cla_1^J$. By  Theorem~\ref{thm:splitting H_1} and our choice of $x_0$, we have $\cla_1^J = \cla_0^J\otimes_{\clr_0}\clr_1$. Therefore, we have $A_1^J=\clr_1/(n\mu-(n-1))((n-1)\mu-n)\clr_1$ as $\clr_1$-modules where $\clr_1=\bbk_1[\mu^{\pm 1}]$ and $\dim_{\bbk_1} \cla_1^J= 2$. By Lemma~\ref{lem:cot knot property}(3), we have $\dim_{\bbk_1} \cla_1^K= 2$. Since $\cla_1^L=\cla_1^K\oplus \cla_1^J$, we have $\dim_{\bbk_1} \cla_1^L =4$.
Since $P_1=P_1^\perp$, we have $\dim_{\bbk_1} P_1 = \frac{1}{2} \dim_{\bbk_1}\cla_1^L$, and therefore $\dim_{\bbk_1} P_1 = 2$. Since $\bl_1^L$, $\bl_1^K$, and $\bl_1^J$ are nonsingular, we see that
\begin{equation}
   \label{eq:ex1}
   P_1\neq \cla_1^K\oplus 0 \text{ and }
   P_1\neq 0\oplus \cla_1^J.
\end{equation}

By Corollary~\ref{cor:cyclic}, the module $P_1$ is cyclic, that is, $P_1=\langle (f(\mu),g(\mu))\rangle$ for some $(f(\mu),g(\mu))\in \cla_1^K\oplus  \cla_1^J$ as right $\clr_1$-modules.
By (\ref{eq:ex1}), we have $f(\mu)\ne 0$ and $g(\mu)\ne 0$.  Since $\dim_{\bbk_1} \cla_1^K = \dim_{\bbk_1} \cla_1^J= 2$, for $\la f(\mu)\ra\subset \cla_1^K$ and $\la g(\mu)\ra \subset \cla_1^J$, we have $\dim_{\bbk_1} \la f(\mu)\ra \in\{1,2\}$,  and $\dim_{\bbk_1} \la g(\mu)\ra \in\{1,2\}$.

Case (1): $\dim_{\bbk_1}\langle f(\mu)\rangle =2$ and $\dim_{\bbk_1}\langle
g(\mu)\rangle =1$.

Then, in $\clr_1$ there exists an element $h(\mu)$ (of degree 1) such that  $f(\mu)h(\mu)\ne 0$ in $\cla_1^K$ and $g(\mu)h(\mu)=0$ in  $\cla_1^J$. Therefore $(f(\mu),g(\mu))\cdot h(\mu) =(f(\mu)h(\mu),0)\in P_1$. Since $\dim_{\bbk_1} \cla_1^K=2$, we have $\dim_{\bbk_1}\langle f(\mu)h(\mu)\rangle\in \{1,2\}$. Since $P_1\neq \cla_1^K\oplus 0$ by (\ref{eq:ex1}), one can see that $\dim_{\bbk_1}\langle f(\mu)h(\mu)\rangle\ne 2$. Therefore, we have  $\dim_{\bbk_1}\langle f(\mu)h(\mu)\rangle =1$. Therefore $\la f(\mu)h(\mu)\ra$ is a nontrivial proper submodule of $\cla_1^K$. By Lemma~\ref{lem:cot knot property}(3), the module $\cla_1^K$ has a unique nontrivial proper submodule $\la \mu-1\ra$, and therefore in $\cla_1^K$ we have $f(\mu)h(\mu)=(\mu-1) a$ for some nonzero $a\in \bbk_1$.  Since $(f(\mu),g(\mu))\in P_1$, $(f(\mu)h(\mu),0)=((\mu-1)a,0)\in P_1$, and $P_1=P_1^\perp$, it follows that 
\[
\bl_1^L( (f(\mu),g(\mu)),( (\mu-1)a,0))=0,
\] 
and hence $\bl_1^K(f(\mu),(\mu-1)a)=0$. This implies that $f(\mu)\in \langle \mu-1\rangle^\perp$. By Lemma~\ref{lem:cot knot property}(3), the submodule $\langle \mu -1\rangle$ is self-annihilating. Therefore, we have $f(\mu)\in\langle \mu-1\rangle$. Since $\dim_{\bbk_1}\langle f(\mu)\rangle =2$ and $\dim_{\bbk_1}\langle \mu -1\rangle =1$, it is a contradiction.

Case (2): $\dim_{\bbk_1}\langle f(\mu)\rangle =1$ and $\dim_{\bbk_1}\langle
g(\mu)\rangle =2$.

Since $\cla_1^K$ has a unique nontrivial proper submodule $\langle \mu -1\rangle$, we have $\la f(\mu)\ra = \la \mu-1\ra$. Therefore we have $f(\mu)=(\mu-1)a$ for some $a\in \bbk_1$. Since $\dim_{\bbk_1}\langle g(\mu)\rangle =2=\dim_{\bbk_1}\cla_1^J$, we have $\la g(\mu) \ra = \cla_1^J$. Since $(f(\mu),g(\mu))\in P_1$ and $\la \mu-1\ra = \la \mu-1\ra^\perp$, we have
\begin{align*}
   0 & =\bl_1^L((f(\mu),g(\mu)),(f(\mu),g(\mu))) \\
   & = \bl_1^L(((\mu-1)a,g(\mu)),((\mu-1)a,g(\mu)))\\
   & =\bl_1^K((\mu-1)a,(\mu-1)a)+\bl_1^J(g(\mu),g(\mu))\\
   & = \bl_1^J(g(\mu),g(\mu)).
\end{align*}
Therefore $g(\mu)\in \la g(\mu)\ra^\perp = (\cla_1^J)^\perp  = 0$, which is a contradiction.

Case (3): $\dim_{\bbk_1}\langle f(\mu)\rangle =2$ and $\dim_{\bbk_1}\langle
g(\mu)\rangle =2$.

For an element $h(\mu)\in \clr_1$, we abuse notation so that $h(\mu)$ also denotes the class in $\cla_1^J$ which is represented by $h(\mu)$. Similarly in $\cla_1^K$. 

Since $\dim_{\bbk_1}\langle
g(\mu)\rangle =2=\dim_{\bbk_1}\cla_1^J$, we have $\langle g(\mu))\rangle = \cla_1^J$, and therefore there exists $j(\mu)\in \clr_1$ such that $g(\mu)j(\mu) =1 $ in $\cla_1^J$. Therefore, there exists $h(\mu)\in \langle f(\mu)\rangle$ such that $(h(\mu),1)\in P_1$. Since $\dim_{\bbk_1}\langle (h(\mu),1)\rangle \ge 2$, this implies that $\langle (h(\mu),1)\rangle =P_1$. By (\ref{eq:ex1}), one sees that $h(\mu)\ne 0$. Therefore $\dim_{\bbk_1}\langle h(\mu)\rangle\in \{1,2\}$. Case (2) implies that $\dim_{\bbk_1}\langle h(\mu)\rangle =2$, since $\dim_{\bbk_1}\langle 1\rangle=2$ where $\langle 1\rangle$ is considered as an $\clr_1$-submodule of $\cla_1^J$. Therefore $\langle h(\mu)\rangle=\cla_1^K$.

Since $\cla_1^K$ has a unique nontrivial proper submodule $\la \mu-1\ra$,  we may assume that in $\cla_1^K$ either $h(\mu)=a$ for some nonzero $a\in \bbk_1$, or $h(\mu)=\mu a+b$ for some $a,b\in \bbk_1$ with $a\ne 0$ and $b\ne -a$.

First suppose $h(\mu)=a$, and hence $P_1 = \la (a,1)\ra$. Let $h'(\mu)= a\Delta_J(\mu)$ in $\cla_1^K$.
Since $\Delta_J(\mu)=0$ in $\cla_1^J$, we have  $(h'(\mu),0)\in P_1$. It follows from (\ref{eq:ex1}) that $\la h'(\mu)\ra \ne \cla_1^K$.

Suppose $h'(\mu)\ne 0$ in $\cla_1^K$. Then $\la h'(\mu)\ra=\la\mu-1\ra$ since $\cla_1^K$ has a unique nontrivial proper submodule $\la \mu-1\ra$ . Therefore, in $\cla_1^K$ we have $h'(\mu) = (\mu-1)b$ for some  $b\in \bbk_1$, and we obtain $(\mu-1,0)\in P_1$. This shows $\bl_1^L( (a,1),(\mu-1,0))=0$, and hence $\bl_1^K(a,\mu-1)=0$. This implies that $a\in \langle \mu-1\rangle^\perp=\langle \mu-1\rangle$, which is a contradiction. Therefore, $h'(\mu)=0$ in $\cla_1^K$. Since $\alpha(m) = m$ for all $m\in \bbz$, in $\cla_1^K$ we have
\[
a( (n-1)\mu-n)(n\mu-(n-1))= a(n\mu-(n-1))( (n-1)\mu -n)=0.
\]
Note that since we are viewing $\cla_1^K$ as a right $\clr_1$-module and $\alpha (a)\ne a$ in general, we cannot cancel $a$ in the above equation.

Recall that $\bbk_1$ is commutative. For $c\in \bbk_1$, we write $c^\mu$ for $\alpha(c)$ for convenience.
Now it follows that in $\cla_1^K=\clr_1/(\mu-1)(\mu-k)\clr_1$ we have
\[
\left( \mu-\frac{n}{n-1}\frac{a}{a^\mu} \right)
 \left( \mu-\frac{n-1}{n}\frac{a^\mu}{a^{\mu^2}} \right)
=
\left( \mu-\frac{n-1}{n}\frac{a}{a^\mu} \right)
 \left( \mu-\frac{n}{n-1}\frac{a^\mu}{a^{\mu^2}} \right)
 =0.
\]
Therefore, in $\clr_1$ we have
\begin{align*}
\left( \mu-\frac{n}{n-1}\frac{a}{a^\mu} \right)
 \left( \mu-\frac{n-1}{n}\frac{a^\mu}{a^{\mu^2}} \right)
 &= \left( \mu-\frac{n-1}{n}\frac{a}{a^\mu} \right)
 \left( \mu-\frac{n}{n-1}\frac{a^\mu}{a^{\mu^2}} \right) \\
 &=(\mu-1)(\mu-k).
\end{align*}
By \cite[Lemma 6.3]{cot03}, it follows that in $\bbk_1$ we have
\[
\frac{n}{n-1}\frac{a}{a^\mu}=\frac{n-1}{n}\frac{a}{a^\mu}=1.
\]
Since $a\ne 0$ and $\alpha$ is an automorphism, it is a contradiction.

Now suppose that $h(\mu)=\mu a + b$ for some $a, b\in \bbk_1$ with $a\ne 0$ and $b\ne -a$. Since $\mu a+ b\not\in \langle \mu -1\rangle$, as in the previous case of $h(\mu)=a$, in $\cla_1^K$ we obtain that
\[
(\mu a+b)(n\mu -(n-1))( (n-1)\mu-n)=
 (\mu a+b)( (n-1)\mu-n)(n\mu-(n-1))=0.
\]
This implies that in $\cla_1^K$ we have
\[
(\mu^2 n a^\mu+\mu(b^\mu n-a(n-1))-b(n-1))( (n-1)\mu-n)=0
\]
and
\[
(\mu^2(n-1)a^\mu + \mu(b^\mu(n-1)-an)-bn) (n\mu-(n-1))=0.
\]
Since $(\mu-1)(\mu-k)=0$ in $\cla_1^K$, we have $\mu^2=\mu(k+1)-k$ and
\[
(\mu(n(k+1)a^\mu+b^\mu n-a(n-1))-(nka^\mu + b(n-1))) ( (n-1)\mu -n)=0
\]
and
\[
(\mu( (n-1)(k+1)a^\mu+b^\mu(n-1)-an)-( (n-1)ka^\mu+bn)) (n\mu-(n-1)) =0.
\]
Again by \cite[Lemma 6.3]{cot03}, one can deduce that in $\bbk_1$ we have
\[
n(k+1)a^\mu + b^\mu n-a(n-1) = nka^\mu + b(n-1)
\]
and
\[
(n-1)(k+1)a^\mu + b^\mu(n-1)-an = (n-1)ka^\mu+bn.
\]
Therefore, in $\bbk_1$ we have  $(a^\mu + b^\mu)n=(a+b)(n-1)$ and $(a^\mu+b^\mu) (n-1) = (a+b)n$. Since $a+b\ne 0$, we have
\[
\frac{n-1}{n}=\frac{a^\mu+b^\mu}{a+b}=\frac{n}{n-1},
\]
which is a contradiction.

Case (4): $\dim_{\bbk_1}\langle f(\mu)\rangle =1$ and $\dim_{\bbk_1}\langle
g(\mu)\rangle =1$.

Since $\dim_{\bbk_1}\langle g(\mu)\rangle =1$, there exists $a\in \bbk_1$ such that $g(\mu)(\mu-a)=0$ in $\cla_1^J$. Since $\dim_{\bbk_1}P_1=2$ and $P_1=\langle(f(\mu),g(\mu))\rangle$, it follows that $f(\mu)(\mu-a)\ne 0$ in $\cla_1^K$ and $(f(\mu)(\mu-a),0)\in P_1$. By (\ref{eq:ex1}), we have $\dim_{\bbk_1}\langle f(\mu)(\mu-a)\rangle =1$. Since $\cla_1^K$ has a unique nontrivial proper submodule $\la \mu-1\ra$,  we obtain that $f(\mu)(\mu-a)=(\mu-1)b$ for some $b\in \bbk_1$. Therefore $(\mu-1,0)\in P_1$, and hence $\mu-1\in P_1^K$. It follows that $P_1^K\ne 0$. By (\ref{eq:ex1}), we have $P_1^K\ne \cla_1^K$. Therefore, $P_1^K$ is a nontrivial proper submodule of $\cla_1^K$, and hence $P_1^K=\langle \mu-1\rangle$.

From Cases (1)--(4), we conclude that $P_1^K=\langle \mu-1\rangle$.
\end{proof}

\subsection*{Acknowledgements}

We wish to thank Jae Choon Cha, Tim Cochran, Chuck Livingston, and Kent Orr for helpful conversations. The first author was supported by the National Research Foundation of Korea(NRF) grant funded by the Korea government(MEST) (No.~2009--0074487). The second author was supported by Basic Science Research Program through the National Research Foundation of Korea (NRF) grant funded by the Ministry of Education, Science and Technology (No.~2012R1A1A001747 and No.~20120000341). This paper was written as part of Konkuk University's research support program for its faculty on sabbatical leave in 2013.

\end{document}